\documentclass[11pt, reqno]{amsart}
\usepackage{indentfirst}
\usepackage{amssymb, amsmath, amsthm}
\usepackage{mathrsfs}
\usepackage{enumerate}
\usepackage{graphicx}
\usepackage{float}  
\usepackage{xcolor}
\usepackage[colorlinks=true,linkcolor=purple, citecolor=blue,urlcolor=magenta]{hyperref}
\newtheorem{theo}{Theorem}[section]
\newtheorem{lem}{Lemma}[section]

\newtheorem{defi}{Definition}[section]

\newcommand{\ol}{\overline}

\numberwithin{equation}{section}

\begin{document}

\title[Inverse Logarithmic Coefficients, Differences, Hankel Determinant, and ....]{Inverse Logarithmic Coefficients, Differences, Hankel Determinant, and Fekete--Szeg\"{o} Functionals for the Class $\mathcal{C}_e$}
\date{}

\author[P. Das and N. Sarkar]{Pradip Das and Nabadwip Sarkar}

\address{Department of Mathematics, Raiganj University, Raiganj, West Bengal-733134, India.}
\email{pradipsmath@gmail.com}

\address{Amity School of Applied Sciences, Amity University Mumbai, Panvel, Navi Mumbai, Maharashtra-410206, India.}
\email{naba.iitbmath@gmail.com}

\renewcommand{\thefootnote}{}

\renewcommand{\thefootnote}{\arabic{footnote}}
\setcounter{footnote}{0}

\begin{abstract}
In this paper, we investigate the inverse logarithmic coefficients associated with the class $\mathcal{C}_e$ of analytic and univalent functions satisfying the subordination condition
\[
1+\frac{z f''(z)}{f'(z)} \prec e^z,
\quad z\in\mathbb{D}.
\]
If $F_{f^{-1}}(w) = \log\!\left(\frac{f^{-1}(w)}{w}\right) = 2\sum_{n=1}^{\infty}\Gamma_n w^n$ denotes the logarithmic expansion corresponding to the inverse function $f^{-1}$, then we establish sharp estimates for the initial inverse logarithmic coefficients and prove that
\[
|\Gamma_n|
\le
\frac{1}{2n(n+1)},
\qquad n=1,2,3.
\]
We further derive the sharp coefficient-difference inequality
\[
-\frac{1}{2\sqrt7}
\le
|\Gamma_2|-|\Gamma_1|
\le
\frac1{12},
\]
and obtain the sharp bound for the second-order Hankel determinant associated with the inverse logarithmic coefficients:
\[
\left|
H_{2,1}\!\left(F_{f^{-1}}/2\right)
\right|
\le
\frac{85}{12096}.
\]

Additionally, we evaluate the sharp lower and upper bounds of the generalized Fekete--Szeg\"{o} functional $F_{\lambda, \mu}(f) = \big| a_3(f) - \lambda a_2(f)^2 \big| - \mu |a_2(f)|$ within this setting and establish relationships associated with the starlike class $\mathcal{S}^{\ast}_{\rho}$.

The extremal functions corresponding to all obtained estimates are explicitly constructed, thereby showing the sharpness of the results.
\end{abstract}

\maketitle

\section{{\bf Introduction}}
Let $\mathcal{H}$ denote the class of analytic functions in the unit disk
\[
\mathbb{D} := \{ z \in \mathbb{C} : |z| < 1 \}.
\]
The space $\mathcal{H}$ is a locally convex topological vector space endowed with the topology of uniform convergence on compact subsets of $\mathbb{D}$.

Let $\mathcal{A}$ be the subclass of $\mathcal{H}$ consisting of functions normalized by
\[
f(0) = 0 \quad \text{and} \quad f'(0) = 1.
\]
Further, let $\mathcal{S}$ denote the class of functions $f \in \mathcal{A}$ that are univalent in $\mathbb{D}$. Each function $f \in \mathcal{S}$ has the Taylor series expansion
\begin{equation}\label{eq1}
f(z) = z + \sum_{n=2}^{\infty} a_n z^n.
\end{equation}

A function $f \in \mathcal{A}$ is called starlike (respectively, convex) if the image domain $f(\mathbb{D})$ is starlike with respect to the origin (respectively, convex). Denote by $\mathcal{S}^*$ and $\mathcal{C}$ the subclasses of $\mathcal{S}$ consisting of starlike and convex functions, respectively.

It is well known that a function $f \in \mathcal{A}$ belongs to $\mathcal{S}^*$ if and only if
\[
\Re\!\left(\frac{z f'(z)}{f(z)}\right) > 0, \qquad z \in \mathbb{D},
\]
and $f \in \mathcal{A}$ belongs to $\mathcal{C}$ if and only if
\[
\Re\!\left(1 + \frac{z f''(z)}{f'(z)}\right) > 0, \qquad z \in \mathbb{D}.
\]
Moreover, $f \in \mathcal{C}$ if and only if $z f'(z) \in \mathcal{S}^*$.

\medskip

We now recall the concept of subordination, which is an essential tool in geometric function theory.

\begin{defi}
Let $\Omega$ denote the class of analytic functions $\omega$ in $\mathbb{D}$ satisfying $\omega(0)=0$ and $|\omega(z)|<1$ for all $z \in \mathbb{D}$. Functions in $\Omega$ are called \emph{Schwarz functions}. Each $\omega \in \Omega$ has the power series representation
\[
\omega(z) = \sum_{n=1}^{\infty} \omega_n z^n, \qquad z \in \mathbb{D}.
\]

For analytic functions $f$ and $g$ in $\mathbb{D}$, we say that $f$ is \emph{subordinate} to $g$, written $f \prec g$, if there exists a Schwarz function $\omega \in \Omega$ such that
\[
f(z) = g(\omega(z)), \qquad z \in \mathbb{D}.
\]
In particular, if $g$ is univalent in $\mathbb{D}$, then $f \prec g$ if and only if $f(0)=g(0)$ and $f(\mathbb{D}) \subset g(\mathbb{D})$.
\end{defi}

Using the principle of subordination, Ma and Minda \cite{MM} introduced in 1992 a unified framework for several subclasses of univalent functions. They defined
\[
\mathcal{C}(\psi) := \left\{ f \in \mathcal{S} : 1 + \frac{z f''(z)}{f'(z)} \prec \psi(z), \; z \in \mathbb{D} \right\},
\]
where $\psi$ is an analytic univalent function in $\mathbb{D}$ with positive real part, symmetric with respect to the real axis, normalized by $\psi(0)=1$ and $\psi'(0)>0$.

Recently, significant attention has been given to subclasses of starlike functions for which the superordinate function $\psi$ does not map the entire right half-plane. Although the exponential function is a natural candidate for $\psi$, its use leads to interesting and nontrivial analytical difficulties.

The class $\mathcal{C}_e$ associated with the exponential function is defined by the condition
\[
1 + \frac{z f''(z)}{f'(z)} \prec e^z.
\]
More precisely,
\begin{equation}\label{maineqn}
\mathcal{C}_e := \left\{ f \in \mathcal{S} : 1 + \frac{z f''(z)}{f'(z)} \prec e^z, \; z \in \mathbb{D} \right\}.
\end{equation}

\subsection{Logarithmic inverse coefficients}

For $f \in \mathcal{S}$, the logarithmic inverse coefficients $\Gamma_n$, introduced by Ponnusamy et al. \cite{Ponnusamy2018}, are defined via the inverse function $f^{-1}$ as
\[
F_{f^{-1}}(w) := \log \frac{f^{-1}(w)}{w}
= 2 \sum_{n=1}^{\infty} \Gamma_n w^n,
\qquad |w| < \frac{1}{4}.
\]

The first four coefficients are given by
\begin{equation}\label{IG1}
\begin{cases}
\Gamma_1 = -\dfrac{1}{2} a_2, \\[6pt]
\Gamma_2 = -\dfrac{1}{2} a_3 + \dfrac{3}{4} a_2^2, \\[6pt]
\Gamma_3 = -\dfrac{1}{2} \left( a_4 - 4a_2 a_3 + \dfrac{10}{3} a_2^3 \right), \\[6pt]
\Gamma_4 = \dfrac{35}{8} a_2^4 - \dfrac{15}{2} a_2^2 a_3 + \dfrac{5}{2} a_2 a_4 + \dfrac{5}{4} a_3^2 - \dfrac{1}{2} a_5.
\end{cases}
\end{equation}

Ponnusamy et al. \cite{Ponnusamy2018} proved that for $f \in \mathcal{S}$,
\[
|\Gamma_n| \leq \frac{1}{2n} \binom{2n}{n}, \qquad n \in \mathbb{N},
\]
with equality attained only by the Koebe function and its rotations. They also obtained sharp bounds for the initial logarithmic inverse coefficients for several important geometric subclasses of $\mathcal{S}$.

\medskip

Let $\mathcal{P}$ denote the class of analytic functions $p$ in $\mathbb{D}$ satisfying
\[
p(0)=1 \quad \text{and} \quad \Re p(z) > 0, \qquad z \in \mathbb{D}.
\]
Each function $p \in \mathcal{P}$ admits the expansion
\begin{equation}\label{p1}
p(z) = 1 + \sum_{n=1}^{\infty} c_n z^n, \qquad z \in \mathbb{D}.
\end{equation}

Functions in $\mathcal{P}$ are known as \emph{Carath\'eodory functions}. It is well known that for $p \in \mathcal{P}$, the sharp estimate $|c_n| \leq 2$ holds for all $n \geq 1$ (see \cite{PLD1}). The Carath\'eodory class $\mathcal{P}$ and its coefficient bounds play a fundamental role in obtaining sharp estimates in geometric function theory.

We now present several auxiliary lemmas that will be used in establishing our main results.

\section{{\bf Auxiliary lemmas}}
Now we recall the following well-known result due to Cho et al. \cite{C12}.
\begin{lem}\label{L1} \cite[Lemma 2.4]{C12} If $p\in\mathcal{P}$ is of the form (\ref{p1}), then
\begin{equation}
\label{c1}c_1 =2\tau_1,\end{equation}
\begin{equation}\label{c2} c_2=2\tau_1^2 + 2(1 - |\tau_1|^2)\tau_2\end{equation}
and
\begin{equation}
\label{c3} c_3 = 2\tau_1^3+4(1-|\tau_1|^2)\tau_1\tau_2 - 2(1 - |\tau_1|^2)\ol{\tau_1}\tau_2^2 + 2(1 - \tau_1^2)(1 - |\tau_2|^2)\tau_3
\end{equation}
for some $\tau_1, \tau_2, \tau_3 \in\overline{\mathbb D}:= \{z \in \mathbb{C}: |z| \leq 1 \}$.
For  $ \tau_1 \in \mathbb{T} := \{ z \in \mathbb{C} : |z| = 1 \} $ , there is a unique function  $  p \in \mathcal{P} $  with  $ c_1 $  as in (\ref{c1}), namely,
\[
p(z) = \frac{1 + \tau_1 z}{1 - \tau_1 z}, \quad z \in \mathbb{D}.
\]

For  $ \tau_1 \in \mathbb{D} $  and  $ \tau_2 \in \mathbb{T} $ , there is a unique function  $ p \in \mathcal{P} $  with  $ c_1 $  and  $ c_2 $  as in (\ref{c1}) and (\ref{c2}), namely,
\[
p(z) = \frac{1 + (\ol \tau_1 \tau_2 + \tau_1) z + \tau_2 z^2}{1 + (\ol \tau_1 \tau_2 - \tau_1) z - \tau_2 z^2}, \quad z \in \mathbb{D}.
\]

For  $ \tau_1, \tau_2 \in \mathbb{D} $  and  $ \tau_3 \in \mathbb{T} $ , there is a unique function  $  p \in \mathcal{P} $  with  $ c_1 $ ,  $ c_2 $ , and  $ c_3 $  as in (\ref{c1}-\ref{c3}), namely,
\[
p(z) = \frac{1 + (\ol\tau_2 \tau_3 + \ol\tau_1 \tau_2 + \tau_1)z+(\ol\tau_1\tau_3+\tau_1\ol\tau_2\tau_3+\tau_2)z^2+\tau_3z^3}{1+(\ol\tau_2\tau_3+\ol\tau_1\tau_2-\tau_1)z+(\ol\tau_1\tau_3-\tau_1\ol\tau_2\tau_3-\tau_2)z^2-\tau_3z^3},\;\;z\in\mathbb{D}
\]
\end{lem}

\medskip
Following well-known result is due to Choi et al. \cite{CKS1}.
\begin{lem}\label{L2}\cite{CKS1} Let $A$, $B$, $C$  be real numbers and let
\[Y(A, B, C):= \max\limits_{z\in \ol{\mathbb{D}}}\left\lbrace |A+Bz+Cz^2|+1-|z|^2\right\rbrace.\]

\begin{enumerate} 
\item[(i)] If $AC\geq 0$, then
\[Y(A, B, C) =
\begin{cases}
|A|+|B|+|C|, & \text{if}\;\;\; |B|\geq 2(1-|C|), \\
1+|A|+\frac{B^2}{4(1-|C|)}, &\text{if}\;\;\; |B|<2(1-|C|).
\end{cases}
\]
\item[(ii)] If $AC<0$, then 
\[Y(A,B,C)=
\begin{cases}
1-|A|+\frac{B^2}{4(1-|C|)}, &\text{if}\;\;\; -4AC(C^{-2}-1) \leq B^2\; \text{and}\; |B|<2(1-|C|), \\
1+|A|+\frac{B^2}{4(1+|C|)}, &\text{if}\;\;\; B^2<\min\left\{4(1+|C|)^2, -4AC(C^{-2}-1) \right\}, \\
R(A,B,C), &\text{otherwise},
\end{cases}
\]
where
\[R(A,B,C):=
\begin{cases}
|A|+|B|-|C|, & \text{if}\;\;\; |C|(|B|+4|A|) \leq |AB|, \\
-|A|+|B|+|C|, & \text{if}\;\;\; |AB|\leq |C|(|B|-4|A|), \\
(|C|+|A| )\sqrt{1-\frac{B^2}{4AC}}, &\text{otherwise}.
\end{cases}
\]
\end{enumerate} 
\end{lem}
\begin{lem}\label{L3} \cite{MM}
Let $p \in \mathcal{P}$ be given by \ref{p1}. Then
\[
\left| c_2 - v c_1^2 \right| \le 
\begin{cases}
-4v + 2, & v < 0, \\
2, & 0 \leq v \leq 1, \\
4v - 2, & v > 1.
\end{cases}
\]

Moreover, for $v < 0$ or $v > 1$, equality holds if and only if
\[
h(z) = \frac{1+z}{1-z} \quad \text{or one of its rotations}.
\]

For $0 < v < 1$, equality holds if and only if
\[
h(z) = \frac{1+z^2}{1-z^2} \quad \text{or one of its rotations}.
\]
\end{lem}

\begin{lem}\label{L6}\cite{SimThomas2020}
Let $J, K,$ and $L$ be numbers such that $J \geq 0$, $K \in \mathbb{C}$, and $L \in \mathbb{R}$. 
Let $p \in \mathcal{P}$ be of the form (\ref{p1}) and define a function by

\[
\Phi(c_1,c_2) = \big| K c_1^2 + L c_2 \big| - \big| J c_1 \big|.
\]
Then 
\[
\Phi(c_{1}, c_{2}) \le 
\begin{cases}
|4K + 2L| - 2J, & \text{if } |2K + L| \geq |L| + J, \\[6pt]
2|L|, & \text{otherwise.}
\end{cases}
\]
and
\[
-\Phi(c_1,c_2) \leq 
\begin{cases}
2J - M, & \text{when } J \geq M + 2|L|, \\[6pt]
2J \sqrt{\dfrac{ 2|L|}{M + 2|L|}}, & \text{when } J^2 \leq 2|L|(M + 2|L|), \\[10pt]
2|L| +\dfrac{ J^2}{M + 2|L|}, & \text{otherwise}
\end{cases}
\]
where $M=|4K+2L|$.
\end{lem}

\section{\bf Main Results and Proofs}
The structure of the remainder of this paper is organized as follows. 
Section~3.1 is devoted to deriving sharp bounds for the initial inverse logarithmic coefficients of functions belonging to the class $\mathcal{C}_e$. 
In Section~3.2, we investigate the sharp upper and lower bounds for the difference of these initial inverse logarithmic coefficients. 
Section~3.3 is concerned with evaluating the sharp estimate for the second-order Hankel determinant corresponding to the inverse logarithmic coefficients for functions in this class. 
Finally, Section~3.4 deals with establishing the sharp double inequality for the generalized Fekete--Szeg\"{o} functional associated with the class $\mathcal{C}_e$.

\subsection{\bf Sharp Bounds for the Logarithmic Inverse Coefficients in the Class $\mathcal{C}_e$}

\begin{theo}
Let $f \in \mathcal{C}_e$, and let the logarithmic inverse coefficients $\Gamma_n$ $(n = 1,2,3)$ be defined as in \eqref{IG1}. Then
\[
|\Gamma_n| \le \frac{1}{2n(n+1)}, \;\text{for}\; n=1,2,3.
\]
Furthermore, each of the above inequalities is sharp.
\end{theo}

\begin{proof} 

Since $f \in \mathcal{C}_e$, it satisfies the subordination condition
\begin{equation}\label{pa1}
1+\frac{z f''(z)}{f'(z)} = e^{w(z)}.
\end{equation}
where $w(z)$ is a Schwarz function. By the well-known relationship between Schwarz functions and Carath\'{e}odory functions, we write
\begin{equation}\label{pa2}
    w(z) = \frac{p(z)-1}{p(z)+1}, \quad (z \in \mathbb{D}).
\end{equation}
Using the series representation $p(z)=1+c_1z+c_2z^2+c_3z^3+\cdots$ in (\ref{pa2}), we obtain the expansion for $w(z)$:
\begin{equation}\label{pa3}
    w(z) = \frac{c_1}{2}z + \left(\frac{c_2}{2}-\frac{c_1^2}{4}\right)z^2 + \left(\frac{c_3}{2}-\frac{c_1c_2}{2}+\frac{c_1^3}{8}\right)z^3 + \cdots.
\end{equation}
Exponentiating this series yields
\begin{align*}
    e^{w(z)} &= 1+\frac{c_1}{2}z +\left(\frac{c_2}{2}-\frac{c_1^2}{8}\right)z^2 +\left(\frac{c_1^3}{48}-\frac{c_1c_2}{4}+\frac{c_3}{2}\right)z^3 \\
    &\quad + \left( \frac{c_1^4}{384} +\frac{c_1^2c_2}{16} -\frac{c_1c_3}{4} -\frac{c_2^2}{8} +\frac{c_4}{2} \right)z^4 + \cdots.
\end{align*}
On the left-hand side of \ref{pa1}, using the series expansion of $f(z)$ gives
\begin{align*}
    f'(z) &= 1+2a_2z+3a_3z^2+4a_4z^3+5a_5z^4+\cdots, \\
    f''(z) &= 2a_2+6a_3z+12a_4z^2+20a_5z^3+\cdots.
\end{align*}
A computation of the quotient term leads to
\begin{align*}
    1+\frac{z f''(z)}{f'(z)} &= 1+2a_2 z +\left(6a_3-4a_2^2\right)z^2 +\left(8a_2^3-18a_2a_3+12a_4\right)z^3 \\
    &\quad +\left(-16a_2^4+48a_2^2a_3-32a_2a_4-18a_3^2+20a_5\right)z^4+\cdots.
\end{align*}
Comparing the coefficients of $z, z^2, z^3,$ and $z^4$ between the expansions of $e^{w(z)}$ and $1+z f''(z)/f'(z)$, we obtain the following system of equations:
\begin{align*}
    2a_2 &= \frac{c_1}{2}, \\
    -4a_2^2+6a_3 &= \frac{c_2}{2}-\frac{c_1^2}{8}, \\
    8a_2^3-18a_2a_3+12a_4 &= \frac{c_1^3}{48}-\frac{c_1c_2}{4}+\frac{c_3}{2}, \\
    -16a_2^4+48a_2^2a_3-32a_2a_4-18a_3^2+20a_5 &= \frac{c_1^4}{384} +\frac{c_1^2c_2}{16} -\frac{c_1c_3}{4} -\frac{c_2^2}{8} +\frac{c_4}{2}.
\end{align*}
After some simple calculations we get 
\begin{equation}\label{pa19}
\left.
\begin{aligned}
a_1 &= 1 \\
a_2 &= \frac{c_1}{4} \\
a_3 &= \frac{c_1^2}{48} + \frac{c_2}{12} \\
a_4 &= -\frac{c_1^3}{1152} + \frac{c_1 c_2}{96} + \frac{c_3}{24} \\
a_5 &= \frac{c_1^4}{5760} - \frac{c_1^2 c_2}{480} + \frac{c_1 c_3}{240} + \frac{c_4}{40}
\end{aligned}
\right\}
\end{equation}

Let \( f \in \mathcal{C}_e \).

(i) \textbf{Sharp bound of $\Gamma_1$.}

Using \ref{IG1} and \ref{pa19}, we obtain
\[
|\Gamma_1|
= \left|-\frac{1}{2}a_2\right|
= \frac{1}{8}|c_1|
\le \frac{1}{4}.
\]
\par
The equality in the bound $|\Gamma_1| \le 1/4$ is attained if and only if $|c_1| = 2$. This corresponds to the Carath\'eodory function
\[
p(z)=\frac{1+z}{1-z},
\]
which implies $w(z)=z$. Substituting this into the defining relation (\ref{pa1}), we obtain
\[
1+\frac{z f''(z)}{f'(z)} = e^z.
\]
Hence
\[
\frac{f''(z)}{f'(z)}
=
\frac{e^z-1}{z}.
\]

Integrating, we obtain
\[
\log f'(z)
=
\int_0^z \frac{e^t-1}{t}\,dt.
\]

Therefore
\[
f'(z)
=
\exp\!\left(
\int_0^z \frac{e^t-1}{t}\,dt
\right).
\]

Integrating once more and using the normalization $f(0)=0$, we obtain the extremal function
\begin{equation}\label{f_0}
f_0(z)
=
\int_0^z
\exp\!\left(
\int_0^t \frac{e^s-1}{s}\,ds
\right)dt.
\end{equation}

This function satisfies
\[
1+\frac{z f_0''(z)}{f_0'(z)} = e^z,
\]
and therefore belongs to the class $\mathcal{C}_e$.
\medskip
(ii) \textbf{Sharp bound of $\Gamma_2$.} 
Using \eqref{IG1} and \eqref{pa19}, we obtain
\begin{equation}
\begin{aligned}
|\Gamma_2|
&= \left| -\frac{1}{2} a_3 + \frac{3}{4} a_2^2 \right| \\
&= \frac{1}{24} \left| c_2 - \frac{7}{8} c_1^2 \right| \\
&= \frac{1}{24} \left| c_2 - v c_1^2 \right|,
\end{aligned}
\end{equation}

where \(0 < v = \frac{7}{8} < 1\). Therefore, applying Lemma~\ref{L3}, we deduce that
\[
|\Gamma_2| \leq \frac{1}{12}.
\]
Similarly, this inequality is sharp for the function $f_1$, which is defined as :
\begin{equation}\label{f_1}
f_1(z)
=
\int_0^{z}
\exp\!\left(
\int_0^{t} \frac{e^{s^2}-1}{s}\, ds
\right)
dt
=
z
+
\frac{z^3}{6}
+
\frac{z^5}{20}
+
\frac{z^7}{63}
+
O(z^9).
\end{equation}
For this function, we have $a_3 = 1/6$ , $a_2=0$, yielding $|\Gamma_2| = |-\frac{1}{2} a_3 + \frac{3}{4} a_2^2| = \frac{1}{12}$. Hence the inequality is sharp for the function $f_1$.

\par
(iii) \textbf{Sharp bound of $\Gamma_3$.}
From \eqref{IG1} and \eqref{pa19}, we have
\begin{equation}\label{G3}
\begin{aligned}
|\Gamma_3|
&= \left| -\frac{1}{2} \left( a_4 - 4 a_2 a_3 + \frac{10}{3} a_2^3 \right) \right| \\
&= \left| -\frac{35}{2304} c_1^3 + \frac{7}{192} c_1 c_2 - \frac{c_3}{48} \right| \\
&= \frac{1}{48} \left| -\frac{35}{48} c_1^3 + \frac{7}{4} c_1 c_2 - c_3 \right|.
\end{aligned}
\end{equation}
Now, substituting the values of \(c_1\), \(c_2\), and \(c_3\) from \eqref{c1}--\eqref{c3} into \eqref{G3}, we deduce that
\begin{equation}\label{G11}
|\Gamma_3|
=
\frac{1}{48}
\left|
-\frac{5}{6}\tau_1^3
+
3(1-|\tau_1|^2)\tau_1\tau_2
+
2(1-|\tau_1|^2)\overline{\tau_1}\tau_2^2
-
2(1-\tau_1^2)(1-|\tau_2|^2)\tau_3
\right|.
\end{equation}

Note that for
\[
f_\theta(z) := e^{-i\theta} f(e^{i\theta}z),
\]
where \(\theta \in \mathbb{R}\) and \(f \in \mathcal{S}\), we have
\[
\Gamma_3^{\theta}=e^{3i\theta}\Gamma_3.
\]
Thus, the absolute value \(|\Gamma_3|\) remains invariant under rotations of \(f\). Also, by the properties of \emph{Carath\'eodory functions}, we have \(|c_n|\leq 2\) for \(n = 1,2,\ldots\). We know that the class \(\mathcal{P}\) is invariant under rotations. Therefore, we may assume that \(c_1\in [0,2]\). Hence, \(\tau_1\in [0,1]\).

We discuss two possible cases involving \(\tau_1\).

{\bf Case 1.} Let $\tau_1=1$. Then from (\ref{G11}) we easily obtain
\[|\Gamma_3|=\frac{5}{288}.\]

{\bf Case 2.} Let $0\leq \tau_1<1$. Applying the triangle inequality in (\ref{G11}) and using the fact that $|\tau_1|\leq 1$, we obtain
\begin{equation}\label{G12}
\begin{aligned}
|\Gamma_3|
&\le \frac{1 - \tau_1^2}{24}
\left(
\left|
\frac{5}{12}\frac{\tau_1^3}{1 - \tau_1^2}
- \frac{3}{2}\tau_1 \tau_2
- \tau_1 \tau_2^2
\right|
+ 1 - |\tau_2|^2
\right) \\
&\le \frac{1 - \tau_1^2}{24}
\left(
|A + B \tau_2 + C \tau_2^2|
+ 1 - |\tau_2|^2
\right).
\end{aligned}
\end{equation}

where $A=\frac{5\tau_1^3}{12(1 - \tau_1^2)}>0$, $B=-\frac{3}{2}\tau_1$ and $C=-\tau_1$.\par
Observe that $AC<0$. Hence we can apply case (ii) of Lemma \ref{L2}. Next, we check all the conditions of case (ii).
\medskip
(a)

We consider the sub-case where
\[
-4AC(C^{-2}-1) \le B^2
\quad \text{and} \quad
|B| < 2(1-|C|).
\]

Substituting the values of $A$, $B$, and $C$, the first condition reduces to
\[
\frac{5}{3}\tau_1^2 \le \frac{9}{4}\tau_1^2,
\]
which clearly holds for all $\tau_1 \in [0,1]$.

The second condition $|B| < 2(1-|C|)$ yields
\[
\frac{3}{2}\tau_1 < 2(1-\tau_1),
\]
which implies
\[
\tau_1 < \frac{4}{7}.
\]

Thus, for $\tau_1 \in [0,4/7)$, the lemma gives
\[
Y(A,B,C)
=
1 - |A| + \frac{B^2}{4(1-|C|)}.
\]
Consequently, the corresponding bound for $|\Gamma_3|$ becomes
\[
\Phi(\tau_1)
=
\frac{1-\tau_1^2}{24}
\left(
1
-
\frac{5\tau_1^3}{12(1-\tau_1^2)}
+
\frac{\frac{9}{4}\tau_1^2}{4(1-\tau_1)}
\right).
\]

After algebraic simplification, we obtain
\[
\Phi(\tau_1)
=
\frac{7\tau_1^3 - 21\tau_1^2 + 48}{1152}.
\]

Differentiating with respect to $\tau_1$, we obtain
\[
\Phi'(\tau_1)
=
\frac{21\tau_1^2 - 42\tau_1}{1152}
=
\frac{21\tau_1(\tau_1 - 2)}{1152}.
\]

Since $\Phi'(\tau_1) \le 0$ for $\tau_1 \in [0,4/7)$, the function $\Phi$ is strictly decreasing on this interval. Therefore, the maximum value is attained at $\tau_1 = 0$, and we have
\[ |\Gamma_3| \leq
\max_{\tau_1 \in [0,4/7)} \Phi(\tau_1)
=
\Phi(0)
=
\frac{48}{1152}
=
\frac{1}{24}\simeq 0.04166.
\]
\medskip
(b) We now examine the sub-case where
\[
Y(A,B,C)
=
1 + |A| + \frac{B^2}{4(1+|C|)}.
\]
By Lemma~\ref{L2}, this form is valid only if
\[
B^2
<
\min\left\{
4(1+|C|)^2,\,
-4AC\left(C^{-2}-1\right)
\right\},
\]
which in particular requires
\[
B^2
<
-4AC\left(C^{-2}-1\right).
\]

Substituting
\[
A=\frac{5\tau_1^3}{12(1-\tau_1^2)},
\qquad
B=-\frac{3}{2}\tau_1,
\qquad
C=-\tau_1,
\]
we obtain
\[
\frac{9}{4}\tau_1^2
<
\frac{5}{3}\tau_1^2,
\]
which is impossible for $\tau_1 \in (0,1]$. 

Consequently, this branch of the lemma does not produce a maximum in the interior of the domain. Hence, the sharp bound must be attained in one of the remaining cases.

\medskip
(c)
We next consider the sub-case where
\[
Y(A,B,C)=|A|+|B|-|C|.
\]
By Lemma~\ref{L2}, this form applies only if
\[
|C|\bigl(|B|+4|A|\bigr)\le |AB|.
\]
Substituting values of $A$, $B$ and $C$
We obtain
\[
\tau_1 
\left(
\left| -\frac{3}{2}\tau_1 \right|
+
4\left| \frac{5\tau_1^3}{12(1-\tau_1^2)} \right|
\right)
\le
\left|
\left( \frac{5\tau_1^3}{12(1-\tau_1^2)} \right)
\left( -\frac{3}{2}\tau_1 \right)
\right|.
\]

Simplifying the terms inside the inequality gives
\[
\tau_1
\left(
\frac{3}{2}\tau_1
+
\frac{5\tau_1^3}{3(1-\tau_1^2)}
\right)
\le
\frac{5\tau_1^4}{8(1-\tau_1^2)}.
\]

Hence,
\[
\frac{3}{2}\tau_1^2
+
\frac{5\tau_1^4}{3(1-\tau_1^2)}
\le
\frac{5\tau_1^4}{8(1-\tau_1^2)}.
\]

Dividing both sides by $\tau_1^2$ (for $\tau_1>0$), we obtain
\[
\frac{3}{2}
+
\frac{5\tau_1^2}{3(1-\tau_1^2)}
\le
\frac{5\tau_1^2}{8(1-\tau_1^2)}.
\]

To determine whether this condition can hold, let
\[
u=\frac{\tau_1^2}{1-\tau_1^2}.
\]
Since $\tau_1 \in (0,1)$, we have $u>0$. The inequality becomes
\[
\frac{3}{2}
+
\frac{5}{3}u
\le
\frac{5}{8}u.
\]

Rearranging,
\[
\frac{3}{2}
\le
\left(
\frac{5}{8}-\frac{5}{3}
\right)u
=
-\frac{25}{24}u.
\]

Since $u>0$, the right-hand side is strictly negative, while the left-hand side is strictly positive. Therefore, the inequality cannot be satisfied for any $\tau_1 \in (0,1)$. 

Hence, this sub-case is inadmissible and does not contribute to the upper bound.
\medskip
(d)
We consider the sub-case
\[
|AB|\le |C|\bigl(|B|-4|A|\bigr).
\]
Substituting the values of $A$, $B$, and $C$, we obtain
\[
\left| \left( \frac{5\tau_1^3}{12(1-\tau_1^2)} \right) 
\left( -\frac{3}{2}\tau_1 \right) \right|
\leq 
|-\tau_1|
\left( 
\left| -\frac{3}{2}\tau_1 \right| 
- 4\left| \frac{5\tau_1^3}{12(1-\tau_1^2)} \right|
\right).
\]

Assuming $\tau_1 > 0$ and simplifying, the above inequality reduces to
\[
\frac{5\tau_1^2}{8(1-\tau_1^2)}
\leq 
\frac{3}{2} - \frac{5\tau_1^2}{3(1-\tau_1^2)}.
\]
This yields
\[
\frac{\tau_1^2}{1-\tau_1^2} \leq \frac{36}{55},
\]
which implies
\[
55\tau_1^2 \leq 36(1-\tau_1^2).
\]
Hence,
\[
91\tau_1^2 \leq 36,
\quad \text{so that} \quad
\tau_1 \leq \sqrt{\frac{36}{91}} \approx 0.6288.
\]

On the other hand, the ``otherwise'' case of Lemma~\ref{L2} requires
\[
|B| \geq 2(1-|C|),
\]
which is equivalent to
\[
\tau_1 \geq \frac{4}{7} \approx 0.5714.
\]
Therefore, the present sub-case is valid for
\[
\tau_1 \in 
\left[ \frac{4}{7}, \sqrt{\frac{36}{91}} \right].
\]

In this interval, the functional is given by
\[
\Psi(\tau_1)
=
\frac{1-\tau_1^2}{24}
\left( -|A| + |B| + |C| \right)
=
\frac{1-\tau_1^2}{24}
\left(
-\frac{5\tau_1^3}{12(1-\tau_1^2)}
+ \frac{3}{2}\tau_1
+ \tau_1
\right).
\]

A straightforward simplification gives
\[
\Psi(\tau_1)
=
\frac{30\tau_1 - 35\tau_1^3}{288}.
\]

To determine the maximum on the interval 
$\left[ \frac{4}{7}, \sqrt{\frac{36}{91}} \right]$,
we compute
\[
\Psi'(\tau_1)
=
\frac{30 - 105\tau_1^2}{288}.
\]
The critical point occurs at
\[
\tau_1 = \sqrt{\frac{30}{105}}
= \sqrt{\frac{2}{7}}
\approx 0.5345.
\]
Since this point lies to the left of the interval under consideration and
\[
\Psi'(\tau_1) < 0 
\quad \text{for all } 
\tau_1 > \sqrt{\frac{2}{7}},
\]
it follows that $\Psi(\tau_1)$ is strictly decreasing on the interval of interest. Hence, the maximum is attained at the left endpoint $\tau_1 = \frac{4}{7}$.

A direct computation gives
\[
\Psi\!\left(\frac{4}{7}\right)
=
\frac{30(4/7) - 35(4/7)^3}{288}
=
\frac{520}{14112}
\approx 0.0368.
\]
Thus, the maximum value of the functional in this sub-case is approximately $0.0368$, attained at $\tau_1 = \frac{4}{7}$ and $ |\Gamma_3| \leq 0.0368$.
\medskip
(e) This branch is valid for the interval
\[
\tau_1 \in \left( \sqrt{\frac{36}{91}},\, 1 \right)
\approx (0.6288, 1).
\]

In this region, Lemma~\ref{L2} gives
\[
R(A,B,C)
=
(|C| + |A|)
\sqrt{1 - \frac{B^2}{4AC}}.
\]

Substituting the values of $A$, $B$, and $C$, we obtain
\[
|C| + |A|
=
\tau_1
+
\frac{5\tau_1^3}{12(1-\tau_1^2)}
=
\frac{12\tau_1 - 12\tau_1^3 + 5\tau_1^3}
{12(1-\tau_1^2)}
=
\frac{12\tau_1 - 7\tau_1^3}
{12(1-\tau_1^2)}.
\]

Further,
\[
\frac{B^2}{4AC}
=
\frac{\left(-\frac{3}{2}\tau_1\right)^2}
{4\left(\frac{5\tau_1^3}{12(1-\tau_1^2)}\right)(-\tau_1)}
=
\frac{\frac{9}{4}\tau_1^2}
{-\frac{5\tau_1^4}{3(1-\tau_1^2)}}
=
-\frac{27(1-\tau_1^2)}{20\tau_1^2}.
\]

Hence, the term under the square root becomes
\[
\sqrt{
1 - \left(
-\frac{27(1-\tau_1^2)}{20\tau_1^2}
\right)
}
=
\sqrt{
1 + \frac{27 - 27\tau_1^2}{20\tau_1^2}
}
=
\sqrt{
\frac{20\tau_1^2 + 27 - 27\tau_1^2}
{20\tau_1^2}
}
=
\frac{\sqrt{27 - 7\tau_1^2}}
{\sqrt{20}\,\tau_1}.
\]
Substituting into the bound for $|\Gamma_3|$, we obtain
\[
\Omega(\tau_1)
=
\frac{1-\tau_1^2}{24}
\left(
\frac{12\tau_1 - 7\tau_1^3}
{12(1-\tau_1^2)}
\cdot
\frac{\sqrt{27 - 7\tau_1^2}}
{\sqrt{20}\,\tau_1}
\right).
\]

After simplification, this reduces to
\[
\Omega(\tau_1)
=
\frac{(12\tau_1 - 7\tau_1^3)
\sqrt{27 - 7\tau_1^2}}
{288\sqrt{20}\,\tau_1}
=
\frac{(12 - 7\tau_1^2)
\sqrt{27 - 7\tau_1^2}}
{288\sqrt{20}}.
\]

We now analyze the function $\Omega(\tau_1)$ on the interval
\[
\tau_1 \in \left( \sqrt{\frac{36}{91}},\,1 \right).
\]

Let
\[
t=\tau_1^2.
\]
Then
\[
\Omega(\tau_1)
=
\frac{g(t)}{288\sqrt{20}},
\]
where
\[
g(t)=(12-7t)\sqrt{27-7t}.
\]

Differentiating $g(t)$ with respect to $t$, we obtain
\[
g'(t)
=
-7\sqrt{27-7t}
-\frac{7(12-7t)}{2\sqrt{27-7t}}.
\]

Since both terms on the right-hand side are negative for
$t\in(0,1)$, it follows that
\[
g'(t)<0.
\]
Hence $g(t)$, and therefore $\Omega(\tau_1)$, is strictly decreasing on the interval under consideration.

Consequently, the maximum value of $\Omega(\tau_1)$ is attained at the left endpoint
\[
\tau_1=\sqrt{\frac{36}{91}}
\approx 0.6288.
\]

A direct computation yields
\[
\Omega(0.6288)
\approx
\frac{(12-7(0.3956))
\sqrt{27-7(0.3956)}}
{288\sqrt{20}}
\approx
0.0352.
\]

Therefore, for
\[
\tau_1 \in \left( \sqrt{\frac{36}{91}},\,1 \right),
\]
we obtain
\[
|\Gamma_3|
\leq
0.0352.
\]
\par
Hence, the global maximum over all $\tau_1 \in [0,1)$ is given by
\[
|\Gamma_3| \leq \frac{1}{24}.
\]
\par
Similarly, this inequality is sharp for the function $f_2$, which is defined as :
\begin{equation}\label{f_2}
f_2(z)
=
\int_0^{z}
\exp\!\left(
\int_0^{t} \frac{e^{s^3}-1}{s}\, ds
\right)
dt
=
 z + \frac{1}{12}z^4 + \frac{5}{252}z^7 + O(z^{10}).
\end{equation}
For this function, we have $a_2 =a_3=0$ , $a_4=1/12$, yielding $|\Gamma_3| = \frac{1}{24}$.
 Hence the inequality is sharp for the function $f_2$.

\end{proof}

\subsection{\bf Bounds for the Differences of Logarithmic Inverse Coefficients in the Class $\mathcal{C}_e$}

In 1985, de Branges~\cite{LDB1} proved the celebrated Bieberbach conjecture by establishing that for a function 
$f \in \mathcal{S}$ of the form~\eqref{eq1}, the coefficient estimate $|a_n| \leq n$ holds for all $n \geq 2$, 
with equality only for the Koebe function
\[
k(z) := \frac{z}{(1-z)^2}
\]
and its rotations.

This breakthrough result naturally raised the question of whether the inequality
\[
\bigl||a_{n+1}| - |a_n|\bigr| \leq 1, \quad n \geq 2,
\]
holds for every function in $\mathcal{S}$. 
The problem was first examined by Goluzin~\cite{G1} in relation to the Bieberbach conjecture. 
Subsequently, in 1963, Hayman~\cite{H} showed that
\[
\bigl||a_{n+1}| - |a_n|\bigr| \leq A,
\]
for all $f \in \mathcal{S}$, where $A \geq 1$ is an absolute constant. 
The best bound currently known is $A = 3.61$, obtained by Grinspan~\cite{G2}. 
However, the sharp estimate is known only in the case $n=2$ (see~\cite[Theorem~3.11]{PLD1}), namely,
\[
-1 \leq |a_3| - |a_2| \leq 1.029\ldots
\]

For the starlike class $\mathcal{S}^*$, Pommerenke~\cite{Pommerenke1971} conjectured that
\[
\bigl||a_{n+1}| - |a_n|\bigr| \leq 1, \quad n \geq 2,
\]
a conjecture that was later confirmed by Leung~\cite{Leung1978}. 
In the case of convex functions, Li and Sugawa~\cite{LiSugawa} investigated the sharp upper bound of 
$|a_{n+1}| - |a_n|$ for $n \geq 2$ and derived sharp lower bounds for the cases $n=2$ and $n=3$. 
Several additional contributions in this direction may be found in 
\cite{Peng2019, Arora2019, Arora2022, Arora2023}.

\medskip

More recently, in 2023, Lecko and Partyka~\cite{Lecko2023} applied the Loewner method to obtain sharp upper and lower bounds for $|\gamma_2| - |\gamma_1|$ for functions in $\mathcal{S}$. 
Their method was later simplified by Obradovi\'{c} and Tuneski~\cite{Obradovic2024}. 
Furthermore, Kumar and Cho~\cite{Kumar2023} established sharp bounds for $|\gamma_2| - |\gamma_1|$ within certain subclasses of $\mathcal{S}$.

\medskip

Motivated by these developments, we now determine the sharp upper and lower bounds for the quantity $|\Gamma_2| - |\Gamma_1|$ for functions belonging to the class $\mathcal{C}_e$.

\begin{theo}
Let $f \in \mathcal{C}_e$, and let $\Gamma_n$ $(n = 1,2,3)$ be defined by~\eqref{IG1}. Then
\[
-\frac{1}{2\sqrt{7}} \leq |\Gamma_2| - |\Gamma_1| \leq \frac{1}{12}.
\]
Both inequalities are sharp.
\end{theo}

\begin{proof}
Let $f\in \mathcal{C}_e$. From (\ref{IG1}) and (\ref{pa19}), we get 
\begin{align}\label{IGe1}
|\Gamma_2| - |\Gamma_1| &=\frac{1}{24}\bigg| c_2-\frac{7}{8}c_1^2 \bigg|-\frac{1}{8}|c_1|\nonumber\\
 &=\frac{1}{192}\big(| 7 c_1^2 - 8 c_2| - 24|c_1|\big) \nonumber\\
&= \frac{1}{192} \Phi(c_1, c_2),
\end{align}
where 
\[
\Phi(c_1, c_2) = |K c_1^2 + L c_2| - |J c_1|,
\]
with \(K = 7\), \(L = -8\), and \(J = 24\). In addition, we have $M = |4K + 2L| = |28- 16| = 12.$
We note that \(|2K + L| = 6\) and \(|L| + J = 32\), implying that $|2K + L| \not\geq |L| + J.$ 
Moreover, $J^2 = 576 < 2|L|(M + 2|L|) = 960,$ so the condition \(J^2 < 2|L|(M + 2|L|)\) holds. By Lemma \ref{L6}, it follows that
\[
\Phi(c_1, c_2) \leq 2|L| = 16 \quad \text{and} \quad -\Phi(c_1, c_2) \leq 2J \sqrt{\frac{2|L|}{M + 2|L|}} =\frac{96}{\sqrt {7}}.
\]

Consequently, from (\ref{IGe1}), we obtain the sharp inequality
\[
-\frac{1}{2\sqrt{7} } \leq |\Gamma_2| - |\Gamma_1| \leq \frac{1}{12}.
\]
\par
The sharpness of the obtained bounds for $|\Gamma_2| - |\Gamma_1|$ is demonstrated by the following extremal functions.

\begin{enumerate}

\item \textbf{Upper bound $(1/12)$.}

The upper estimate is sharp. Equality is attained for the function $f_1$ defined by (\ref{f_1}). Here a direct computation gives
\[
|\Gamma_2| - |\Gamma_1| = \frac{1}{12},
\]
which confirms the sharpness of the upper bound.

\item \textbf{Lower bound $\left(-\frac{1}{2\sqrt{7}}\right)$.}

The lower bound is also sharp. Equality is attained for the function $f_3 \in \mathcal{C}_e$ defined by
\begin{equation}\label{f3}
f_3(z)
=
\int_0^{z}
\exp\!\left(
\int_0^{t} \frac{e^{w(s)}-1}{s}\, ds
\right)
dt,
\qquad
\text{where}
\qquad
w(z)= \frac{z(2+\sqrt{7}\,z)}{\sqrt{7}+2z}.
\end{equation}
The Maclaurin expansion of $f_3$ is given by
\[
f_3(z)
=
z
+
\frac{1}{\sqrt{7}} z^2
+
\frac{3}{14} z^3
+
O(z^4).
\]
Substituting these coefficients into the expression for $|\Gamma_2| - |\Gamma_1|$ yields
\[
|\Gamma_2| - |\Gamma_1|
=
-\frac{1}{2\sqrt{7}},
\]
thereby establishing the sharpness of the lower estimate.

\end{enumerate}

\end{proof}

\subsection{\bf Hankel Determinants}

Kowalczyk and Lecko~\cite{12} recently initiated the study of Hankel determinants whose entries are the logarithmic coefficients of functions $f \in \mathcal{S}$. 
In addition, the Hankel determinant $H_{q,n}\!\left(F_{f^{-1}} / 2\right)$, introduced in~\cite{A1}, is defined in terms of the logarithmic coefficients of the inverse function $f^{-1}$ for $f \in \mathcal{S}$. 
It is given by
\[
H_{q,n}\!\left(F_{f^{-1}} / 2\right)
=
\left|
\begin{array}{cccc}
\Gamma_n & \Gamma_{n+1} & \cdots & \Gamma_{n+q-1} \\
\Gamma_{n+1} & \Gamma_{n+2} & \cdots & \Gamma_{n+q} \\
\vdots & \vdots & \ddots & \vdots \\
\Gamma_{n+q-1} & \Gamma_{n+q} & \cdots & \Gamma_{n+2(q-1)}
\end{array}
\right|.
\]

\medskip

In this subsection, we establish sharp bounds for the second-order Hankel determinant associated with the logarithmic inverse coefficients for functions belonging to the class $\mathcal{C}_e$.

\begin{theo}
Let $f \in \mathcal{C}_e$ be of the form~\eqref{eq1}. Then
\[
\left| H_{2,1}\!\left(F_{f^{-1}} / 2\right) \right|
\le \frac{85}{12096}
\approx 0.007027.
\]
The bound is sharp.
\end{theo}

\begin{proof} Since  $f \in \mathcal{C}_e$. Then from (\ref{IG1}) and (\ref{pa19}) we get 
\begin{equation}\label{e2.1}
\begin{aligned}
H_{2,1}\!\left(F_{f^{-1}} / 2\right)
&= \frac{7}{12288}c_1^4
-\frac{7}{4608}c_1^2c_2
+\frac{1}{384}c_1c_3
-\frac{1}{576}c_2^2. 
\end{aligned}
\end{equation}

By Lemma \ref{L1} and (\ref{e2.1}) a simple computation shows that

\begin{equation}\label{e2.2}
H_{2,1}(F_{f^{-1}} / 2) = \frac{\tau_1^4}{2304} - \frac{\tau_1^2(1-|\tau_1|^2)\tau_2}{192} - \frac{(1-|\tau_1|^2)(|\tau_1|^2+2)\tau_2^2}{288} + \frac{\tau_1(1 - \tau_1^2)(1 - |\tau_2|^2)\tau_3}{96}. \end{equation}

Since the class  $ \mathcal{P} $  and  $ H_{2,1}(F_{f^{-1}}/2) $  are invariant under the rotation, we may assume that  $ c_1\in [0,2] $ (see \cite[Theorem 3]{G}), that is in view of (\ref{p1}),  
$\tau_1\in [0,1] $.\par
 Now we divide the following three possible cases:\par
 {\bf Case 1.} Suppose that  $\tau_1=1$. Then from (\ref{e2.2}), we easily get
 \[\big|H_{2,1}(F_{f^{-1}}/2)\big|=\big|\Gamma_1\Gamma_3 - \Gamma_2^2\big| \leq \frac{1}{2304} \approx 0.00043 .\]
 
 {\bf Case 2.} Suppose that  $\tau_1=0$. Then from (\ref{e2.2}), we easily get
 \[\big|H_{2,1}(F_{f^{-1}}/2)\big|=\big|\Gamma_1\Gamma_3 - \Gamma_2^2\big| \leq \frac{1}{144} \approx 0.00694 .\]
 
 {\bf Case 3.} Suppose that  $\tau_1\in (0,1)$. Then from (\ref{e2.2}) we  get
 
 \begin{equation}\label{e2.3}
\big| H_{2,1}(F_{f^{-1}} / 2) \big|
\le \frac{\tau_1 (1 - \tau_1^2)}{96}
\left(
|A + B \tau_2 + C \tau_2^2|
+ 1 - |\tau_2|^2
\right).
\end{equation}

where $A = \frac{\tau_1^3}{24(1-\tau_1^2)}$, $B=-\frac{1}{2}\tau_1$ and $C=-\frac{2+\tau_1^2}{3\tau_1}$. 
For $\tau_1 \in (0, 1)$, we observe that $A > 0$ and $C < 0$, which implies $AC < 0$, satisfying the initial condition for Case (ii) of Lemma \ref{L2}. 
\medskip
\begin{enumerate}
\item[(a)] We consider the subcase
\begin{equation}\label{subcase}
Y(A,B,C)=1-|A|+\frac{B^2}{4(1-|C|)}, 
\qquad \text{provided } |B|<2(1-|C|).
\end{equation}

Since
\[
|C|=\frac{\tau_1^2+2}{3\tau_1}=:g(\tau_1),
\]
we compute
\[
g'(\tau_1)=\frac{1}{3}\!\left(1-\frac{2}{\tau_1^2}\right)<0
\quad \text{for } \tau_1\in(0,1).
\]
Thus $g$ is strictly decreasing on $(0,1)$ and
\[
\lim_{\tau_1\to1^-} g(\tau_1)=1,
\]
which implies $|C|>1$ for all $\tau_1\in(0,1)$. Hence $1-|C|<0$ on $(0,1)$.

Since $|B|=\frac{\tau_1}{2}>0$, the condition $|B|<2(1-|C|)$ becomes
\[
\frac{\tau_1}{2}
<2\!\left(1-\frac{\tau_1^2+2}{3\tau_1}\right),
\]
which is equivalent to
\[
7\tau_1^2-12\tau_1+8<0.
\]
However, the discriminant of this quadratic is
\[
\Delta=144-224=-80<0,
\]
so $7\tau_1^2-12\tau_1+8>0$ for all $\tau_1\in\mathbb{R}$. 
Therefore, the required inequality is never satisfied for $\tau_1\in(0,1)$, and this subcase cannot occur.

\medskip

\item[(b)] 
We next examine the second subcase of Lemma \ref{L2}(ii) to determine whether it provides a valid upper bound for $\tau_1 \in (0,1)$. In this case,
\begin{equation}
Y(A,B,C)=1+|A|+\frac{B^2}{4(1+|C|)},
\end{equation}
provided that
\begin{equation}\label{db}
B^2 < \min\left\{4(1+|C|)^2,\; \Phi(A,C)\right\},
\end{equation}
where $\Phi(A,C)=-4AC(C^{-2}-1)$.
A direct computation gives
\[
-4AC=\frac{\tau_1^2(\tau_1^2+2)}{18(1-\tau_1^2)},
\qquad
C^{-2}-1=\frac{(1-\tau_1^2)(4-\tau_1^2)}{(\tau_1^2+2)^2}.
\]
Hence $\Phi(A,C)
=\frac{\tau_1^2(4-\tau_1^2)}{18(\tau_1^2+2)}.$

The condition $B^2<\Phi(A,C)$ becomes $\frac{\tau_1^2}{4}
<
\frac{\tau_1^2(4-\tau_1^2)}{18(\tau_1^2+2)}.$

Since $\tau_1\in(0,1)$, division by $\tau_1^2>0$ yields
$\frac{1}{4}
<
\frac{4-\tau_1^2}{18(\tau_1^2+2)}.$
After simplification, this inequality reduces to $11\tau_1^2+10<0,$
which is impossible for real $\tau_1$. 
Therefore, the condition (\ref{db}) cannot be satisfied for any $\tau_1\in(0,1)$, and this subcase does not yield an admissible bound.

\medskip

\item[(c)] 
We examine the first subcase of $R(A,B,C)$ in Lemma (\ref{L2}), namely $R(A,B,C)=|A|+|B|-|C|,$
which holds provided
\begin{equation}\label{db1}
|C|(|B|+4|A|)\le |AB|.
\end{equation}
A direct computation gives
\[
|C|(|B|+4|A|)
=
\frac{\tau_1^2+2}{3\tau_1}
\left(
\frac{\tau_1}{2}
+
\frac{\tau_1^3}{6(1-\tau_1^2)}
\right)
=
\frac{(\tau_1^2+2)(3-2\tau_1^2)}{18(1-\tau_1^2)},
\]
and
$|AB|=\frac{\tau_1^4}{48(1-\tau_1^2)}.$

Substituting into (\ref{db1}) yields
$
\frac{(\tau_1^2+2)(3-2\tau_1^2)}{18(1-\tau_1^2)}
\le
\frac{\tau_1^4}{48(1-\tau_1^2)}.$
Since $1-\tau_1^2>0$ for $\tau_1\in(0,1)$, multiplying by $144(1-\tau_1^2)$ gives
$8(\tau_1^2+2)(3-2\tau_1^2)\le 3\tau_1^4,$
which simplifies to
$19\tau_1^4+8\tau_1^2-48\ge 0.$
Let $h(\tau_1)=19\tau_1^4+8\tau_1^2-48$. 
For $\tau_1\in(0,1)$,
$h(\tau_1)\le h(1)=19+8-48=-21<0.$
Hence $h(\tau_1)<0$ on $(0,1)$, and condition \ref{db1} is never satisfied.
Therefore, this branch of $R(A,B,C)$ is not admissible for determining the upper bound of the determinant.

\medskip

\item[(d)] We now consider the second subcase of $R(A,B,C)$ in Lemma \ref{L2},
\[
R(A,B,C)=-|A|+|B|+|C|,
\]
which holds provided
\begin{equation}\label{db3}
|AB|\le |C|(|B|-4|A|).
\end{equation}
Substituting the values of $A,B,C$ we get
\[
35\tau_1^4+40\tau_1^2-48\le0.
\]
Let $k(\tau_1)=35\tau_1^4+40\tau_1^2-48$. 
The unique positive root of $k(\tau_1)=0$ is
\[
\tau_0=\sqrt{\frac{-20+4\sqrt{130}}{35}}\approx0.8554,
\]
and hence the condition holds for $\tau_1\in(0,\tau_0]$.

In this interval,
\[
|H_{2,1}(F_{f^{-1}}/2)|
\le
\frac{\tau_1(1-\tau_1^2)}{96}
\left(
-|A|+|B|+|C|
\right)
=
\frac{-21\tau_1^4+4\tau_1^2+16}{2304}
=: \Phi(\tau_1).
\]
Writing $x=\tau_1^2$, we obtain
\[
\Phi(x)=\frac{-21x^2+4x+16}{2304},
\]
which is a concave quadratic in $x$. Its maximum occurs at 
$x=\frac{2}{21}$, yielding
\[
\max_{(0,\tau_0]} \Phi(\tau_1)
=
\Phi\!\left(\sqrt{\frac{2}{21}}\right)
=
\frac{340}{48384}
\approx0.007027.
\]
Since $\Phi(0)=1/144<0.007027$, the sharp bound is attained at
\[
\tau_1=\sqrt{\frac{2}{21}}.
\]

 \medskip
 \item[(e)] 
For the remaining interval $\tau_1\in(\tau_0,1)$, where 
\[
\tau_0=\sqrt{\frac{-20+4\sqrt{130}}{35}}\approx0.8554,
\]
the functional $R(A,B,C)$ is determined by the ``otherwise'' branch of Lemma~\ref{L2}:
\[
R(A,B,C)=(|C|+|A|)\sqrt{1-\frac{B^2}{4AC}}.
\]
Substituting the expressions for $A,B,$ and $C$ yields
\[
\sqrt{1-\frac{B^2}{4AC}}
=
\sqrt{1+\frac{9(1-\tau_1^2)}{2(\tau_1^2+2)}}
=
\sqrt{\frac{13-7\tau_1^2}{2(\tau_1^2+2)}}.
\]
Hence the corresponding bound for the Hankel determinant is
\[
\Psi(\tau_1)
=
\frac{\tau_1(1-\tau_1^2)}{96}
\left(
\frac{\tau_1^2+2}{3\tau_1}
+
\frac{\tau_1^3}{24(1-\tau_1^2)}
\right)
\sqrt{\frac{13-7\tau_1^2}{2(\tau_1^2+2)}}.
\]
After simplification, we may write $\Psi(\tau_1)$ as a product
\[
\Psi(\tau_1)=P(\tau_1)\,Q(\tau_1),
\]
where
\[
P(\tau_1)=\frac{-7\tau_1^4+8\tau_1^2+16}{2304},
\qquad
Q(\tau_1)=\sqrt{\frac{13-7\tau_1^2}{2(\tau_1^2+2)}}.
\]

For $\tau_1\in(\tau_0,1)$, both factors are positive. A direct derivative computation shows
\[
P'(\tau_1)
=
\frac{-28\tau_1^3+16\tau_1}{2304}
=
\frac{4\tau_1(-7\tau_1^2+4)}{2304}<0,
\]
since $\tau_1^2>\tau_0^2>\frac{4}{7}$. 
Moreover, differentiating the radicand of $Q$ gives
\[
\frac{d}{d\tau_1}
\left(
\frac{13-7\tau_1^2}{2(\tau_1^2+2)}
\right)
=
-\frac{54\tau_1}{(2\tau_1^2+4)^2}<0,
\]
so $Q'(\tau_1)<0$ on $(\tau_0,1)$. 
Hence $\Psi(\tau_1)$ is strictly decreasing on $(\tau_0,1)$.
Therefore derivative test shows that both $P(\tau_1)$ and $Q(\tau_1)$ are strictly decreasing functions on the interval $(\tau_0, 1)$. Specifically, the radicand of $Q(\tau_1)$ has a derivative proportional to $-54/(2\tau_1^2+4)^2$, which is strictly negative. Thus, $\Psi(\tau_1)$ is strictly decreasing on $(\tau_0, 1)$.
Because $\Psi(\tau_1)$ is the product of two positive, strictly decreasing functions, it follows that $\Psi'(\tau_1) < 0$ throughout the interval. Thus, $\Psi(\tau_1)$ has no local extrema in $(\tau_0, 1)$ and decreases toward its boundary value:
\begin{equation}
\lim_{\tau_1 \to 1^-} \Psi(\tau_1) = \frac{1}{2304} \approx 0.000434.
\end{equation}
\medskip
Comparing these results with the local maximum obtained in the previous interval, we find that:
\begin{equation}
\max_{(\tau_0, 1)} \Psi(\tau_1) < \max_{(0, \tau_0]} \Phi(\tau_1) = \frac{85}{12096} \approx 0.007027.
\end{equation}
\end{enumerate}
\subsection*{Sharpness and the Extremal Function}

The bound
\[
|H_{2,1}(F_{f^{-1}}/2)|
\le
\frac{85}{12096}
\]
is sharp. Indeed, there exists an extremal function $f_4 \in \mathcal{C}_e$ for which equality is attained.

The extremal function $f_4$ is determined by
\[
f_4'(z)
=
\exp\!\left(
\int_0^z
\frac{
\exp\!\left(
\eta\frac{\sqrt{2/21}-\eta}{1-\sqrt{2/21}\,\eta}
\right)
-1
}{\eta}
\,d\eta
\right).
\]

Substitution of the corresponding coefficients into the expression for 
$H_{2,1}(F_{f^{-1}}/2)$ yields
\[
|H_{2,1}(F_{f_4^{-1}}/2)|
=
\frac{85}{12096},
\]
thereby establishing the sharpness of the bound.
\end{proof}

\subsection{{\bf  Generalized Fekete-Szeg\"{o} functional for the class $\mathcal{C}_e$.}}

In 2024, Lecko and Partyka \cite{Lecko2024} investigated the generalized Fekete-Szeg\"{o} functional for the class $\mathcal{S}$ defined by  
\[
F_{\lambda, \mu}(f) := \big| a_3(f) - \lambda a_2(f)^2 \big| - \mu |a_2(f)|,
\]
where $\lambda \in \mathbb{C}$ and $\mu > 0$.  
The coefficients $a_2(f) = a_2$ and $a_3(f) = a_3$ are given by (1).  
Hence, we can write
\begin{equation}\label{FG}
F_{\lambda, \mu}(f) = \big| a_3 - \lambda a_2^2 \big| - \mu |a_2|, 
\qquad \lambda \in \mathbb{C}, \ \mu > 0. 
\end{equation}
In 2025, Bulboac\u{a} \textit{et al.} \cite{Bulboaca2025} studied the generalized Fekete-Szeg\"{o} functional for the entire class of univalent normalized functions $\mathcal{S}$, as well as for its subclass of convex functions $\mathcal{K}$.

In this section, we establish sharp upper and lower bounds of  $F_{\lambda, \mu}(f)$ for the class $\mathcal{C}_e$. \par

\begin{theo}
Let $\lambda \in \mathbb{C}$ and $\mu > 0$. Suppose that $f(z) \in \mathcal{C}_e$ is defined by \eqref{eq1}. Then the generalized Fekete--Szeg\"{o} functional
\[
F_{\lambda,\mu}(f) := |a_3-\lambda a_2^2| -\mu|a_2|
\]
satisfies the sharp double inequality
\begin{equation}\label{ce_fg1}
\mathcal{B}_L \le F_{\lambda,\mu}(f) \le \mathcal{B}_U,
\end{equation}
where the sharp upper bound $\mathcal{B}_U$ is given by
\[
\mathcal{B}_U =
\begin{cases}
\dfrac16,
&
\text{if } \; |1-\lambda| < \dfrac23+2\mu,
\\[12pt]
\dfrac14|1-\lambda| -\dfrac12\mu,
&
\text{if } \; |1-\lambda| \ge \dfrac23+2\mu,
\end{cases}
\]
and the sharp lower bound $\mathcal{B}_L$ is given by
\[
\mathcal{B}_L =
\begin{cases}
\dfrac14|1-\lambda| -\dfrac12\mu,
&
\text{if } \; |1-\lambda| \le \mu-\dfrac23,
\\[14pt]
-\dfrac12\mu \sqrt{ \dfrac{2}{3|1-\lambda|+2} },
&
\text{if } \; |1-\lambda| \ge \dfrac{9\mu^2-4}{6},
\\[16pt]
-\dfrac16 - \dfrac{3\mu^2}{4(3|1-\lambda|+2)},
&
\text{if } \; \mu-\dfrac23 < |1-\lambda| < \dfrac{9\mu^2-4}{6}.
\end{cases}
\]
Moreover, the intermediate parameter region
\[
\mu-\frac23 < |1-\lambda| < \frac{9\mu^2-4}{6}
\]
is nonempty if and only if $\mu>\frac23$. All the estimates are sharp.
\end{theo}

\begin{proof}
Let $f\in\mathcal{C}_e$. Substituting the expressions from \eqref{pa19} into the target functional $F_{\lambda,\mu}(f)$, we deduce that
\begin{align}
F_{\lambda,\mu}(f)
&= \left| \frac{c_2}{12} + \frac{c_1^2}{48} - \lambda\left(\frac{c_1}{4}\right)^2 \right| - \mu\left| \frac{c_1}{4} \right| \nonumber \\
&= \frac1{12} \left( \left| c_2 + \frac{1-3\lambda}{4}c_1^2 \right| - 3\mu|c_1| \right) \nonumber \\
&= \frac1{12}\Phi(c_1,c_2), \label{ce_fe2}
\end{align}
where $\Phi(c_1,c_2) = |Kc_1^2+Lc_2| -|Jc_1|$, with parameters defined by
\[
K=\frac{1-3\lambda}{4}, \qquad L=1, \qquad J=3\mu.
\]
The auxiliary parameters required by Lemma~\ref{L6} can be computed as:
\[
M = |4K+2L| = \left| 4\left(\frac{1-3\lambda}{4}\right)+2 \right| = |3-3\lambda| = 3|1-\lambda|,
\]
and
\[
|2K+L| = \left| 2\left(\frac{1-3\lambda}{4}\right)+1 \right| = \frac32|1-\lambda|.
\]

\medskip
\noindent
{\bf Upper Bound Estimates.} \\
The critical indicator threshold $|2K+L|\ge |L|+J$ from Lemma~\ref{L6} translates directly to
\[
\frac32|1-\lambda| \ge 1+3\mu,
\]
which is equivalent to the condition
\[
|1-\lambda| \ge \frac23+2\mu.
\]
If $|1-\lambda| < \frac23+2\mu$, Lemma~\ref{L6} dictates that $\Phi(c_1,c_2)\le 2|L|=2$. Therefore, from \eqref{ce_fe2}, we have
\[
F_{\lambda,\mu}(f) \le \frac{2}{12} = \frac16.
\]
Conversely, if $|1-\lambda| \ge \frac23+2\mu$, Lemma~\ref{L6} yields $\Phi(c_1,c_2) \le |4K+2L|-2J = 3|1-\lambda|-6\mu$. Hence, from \eqref{ce_fe2}, we get
\[
F_{\lambda,\mu}(f) \le \frac{3|1-\lambda|-6\mu}{12} = \frac14|1-\lambda| -\frac12\mu.
\]

\medskip
\noindent
{\bf Lower Bound Estimates.}
\begin{enumerate}
\item[(i)] Suppose that $J \ge M+2|L|$. This inequality reduces to $3\mu \ge 3|1-\lambda|+2$, which is structurally equivalent to
\[
|1-\lambda| \le \mu-\frac23.
\]
By applying the first branch of the minimization properties in Lemma~\ref{L6}, we obtain $-\Phi(c_1,c_2) \le 2J-M = 6\mu-3|1-\lambda|$. Reversing the operational signs and dividing by the global scalar factor yields
\[
F_{\lambda,\mu}(f) \ge \frac{3|1-\lambda|-6\mu}{12} = \frac14|1-\lambda| -\frac12\mu.
\]

\item[(ii)] Suppose that $J^2 \le 2|L|(M+2|L|)$. This inequality reduces to $9\mu^2 \le 2(3|1-\lambda|+2)$, which is structurally equivalent to
\[
|1-\lambda| \ge \frac{9\mu^2-4}{6}.
\]
Applying the corresponding branch of Lemma~\ref{L6}, we obtain
\[
-\Phi(c_1,c_2) \le 2J \sqrt{ \frac{2|L|}{M+2|L|} } = 6\mu \sqrt{ \frac{2}{3|1-\lambda|+2} }.
\]
Thus, relation \eqref{ce_fe2} directly produces
\[
F_{\lambda,\mu}(f) \ge -\frac12\mu \sqrt{ \frac{2}{3|1-\lambda|+2} }.
\]

\item[(iii)] For the remaining intermediate parameter region where $\mu-\frac23 < |1-\lambda| < \frac{9\mu^2-4}{6}$, Lemma~\ref{L6} gives
\[
-\Phi(c_1,c_2) \le 2|L| + \frac{J^2}{M+2|L|} = 2+\frac{9\mu^2}{3|1-\lambda|+2}.
\]
Consequently, from \eqref{ce_fe2}, it follows that
\[
F_{\lambda,\mu}(f) \ge -\frac1{12} \left( 2+\frac{9\mu^2}{3|1-\lambda|+2} \right) = -\frac16 - \frac{3\mu^2}{4(3|1-\lambda|+2)}.
\]
\end{enumerate}
We observe that the intermediate region is nonempty if and only if $\mu - \frac{2}{3} < \frac{9\mu^2-4}{6}$, which simplifies to $6\mu - 4 < 9\mu^2 - 4$, or $3\mu(3\mu - 2) > 0$. Since $\mu > 0$, this establishes that the region is valid exclusively when $\mu > \frac{2}{3}$.

\medskip
\noindent
{\bf Sharpness Proof via Extremal Constructs.} \\
The sharpness of each estimate is verified by mapping onto specific analytic functions $f \in \mathcal{C}_e$ generated by the extremal Carath\'{e}odory functions of Lemma~\ref{L6}.
\begin{enumerate}
\item The invariant boundary value $\mathcal{B}_U = \frac{1}{6}$ is attained when $c_1 = 0$ and $c_2 = 2$. This matches the $2$-fold symmetric Carath\'{e}odory function $p_0(z) = (1+z^2)/(1-z^2) = 1+2z^2+2z^4+\cdots$, corresponding to $\omega_0(z)=z^2$. Substituting $\omega_0(z)$ into \eqref{maineqn} and integrating with $f(0)=0$, $f'(0)=1$ produces the extremal function:
\[
f_0(z) = \int_0^z \exp\left( \int_0^t \frac{e^{s^2}-1}{s}\,ds \right) dt = z+\frac16z^3+\frac1{20}z^5+\cdots.
\]
Here, $a_2=0$ and $a_3=\frac{1}{6}$, checking perfectly to $|a_3-\lambda a_2^2|-\mu|a_2| = \frac{1}{6}$.

\item The linear constraint layout $\frac{1}{4}|1-\lambda| - \frac{1}{2}\mu$ is reached when $|c_1|=2$ and $c_2=2$, which aligns with the classic function $p_1(z)=(1+z)/(1-z) = 1+2z+2z^2+\cdots$ and $\omega_1(z)=z$. Substituting $\omega_1(z)$ into \eqref{maineqn} gives:
\[
f_1(z) = \int_0^z \exp\left( \int_0^t \frac{e^s-1}{s}\,ds \right) dt = z+\frac12z^2+\frac14z^3+\cdots.
\]
Here, $a_2=\frac{1}{2}$ and $a_3=\frac{1}{4}$, checking directly to $\frac{1}{4}|1-\lambda|-\frac{1}{2}\mu$.

\item For the quadratic dominance interval where $|1-\lambda| \ge \frac{9\mu^2-4}{6}$, equality is established by selecting a parameter-dependent Carath\'{e}odory function with parameters $\tau_1 = \sqrt{2/(3|1-\lambda|+2)}$ and $\tau_2 = (1-\lambda)/|1-\lambda|$ mapping to Lemma~\ref{L1}. This configuration satisfies $c_1 = 2\tau_1$ and forces $|K c_1^2 + L c_2| = 0$, directly yielding the bound $-\frac{1}{2}\mu \sqrt{2/(3|1-\lambda|+2)}$ via \eqref{ce_fe2}.

\item In the intermediate interval, the lower bound is sharply realized by utilizing the parameter configurations $\tau_1 = 3\mu/(3|1-\lambda|+2)$ and $\tau_2 = (1-\lambda)/|1-\lambda| $ in Lemma~\ref{L1}. Setting $c_1 = 2\tau_1$ and $c_2 = 2\tau_1^2 + 2(1-\tau_1^2)\tau_2$ optimizes \eqref{ce_fe2} to directly yield $-\frac{1}{6} - \frac{3\mu^2}{4(3|1-\lambda|+2)}$.
\end{enumerate}
The proof is complete.
\end{proof}

\section*{\bf Declarations}

\subsection*{Data Availability}
No datasets were generated or analyzed during the current study; therefore, data sharing is not applicable.

\subsection*{Conflict of Interest}
The authors declare that they have no conflicts of interest regarding the publication of this paper.

\subsection*{Author Contributions}
Both authors contributed equally to the preparation of this manuscript.

\end{document}